\documentclass[11pt]{amsart}

\usepackage{amssymb,amsmath,amscd}
\usepackage{mathtools}
\usepackage{verbatim}

\numberwithin{equation}{section}

\newtheorem{theorem}{Theorem}[section]
\newtheorem{corollary}[theorem]{Corollary}
\newtheorem{lemma}[theorem]{Lemma}
\newtheorem{proposition}[theorem]{Proposition}
\theoremstyle{definition}
\newtheorem{example}[theorem]{Example}

\newtheorem{definition}[theorem]{Definition}
\newtheorem{question}[theorem]{Question}

\newcommand{\tend}[1]{\mbox{}\nolinebreak[4]\hfill$\vartriangle$\end{#1}}
\newcommand{\N}{\mathbb{N}}		
\newcommand{\C}{\mathbb{C}}
\newcommand{\fbd}{\mathrm{fbd}}
\newcommand{\rank}{\mathrm{rank}}
\newcommand{\reg}{\mathrm{reg}}
\newcommand{\st}{\ |\ }
\newcommand{\FFS}{\phi_\mathrm{s}}

\title{On topological classification of complex mappings}



\author[H. Seyedinejad]{Hadi Seyedinejad}
\address{Department of Mathematics, Western University, London, Ontario \linebreak N6A 5B7, Canada}
\email{sseyedin@alumni.uwo.ca}

\begin{document}

\begin{abstract}
We study the topological invariant $\phi$ of Kwieci\'nski and Tworzewski, particularly beyond the case of mappings with smooth targets. We derive a lower bound for $\phi$ of a general mapping, which is similarly effective as the upper bound given by Kwieci\'nski and Tworzewski. Some classes of mappings are identified for which the exact value of $\phi$ can be computed. Also, we prove that the variation of $\phi$ on the source space of a mapping with a smooth target is semicontinuous in Zariski topology.
\end{abstract}

\maketitle

\section{Introduction and main results}

We study the topology of complex analytic mappings (which we simply call \emph{mappings}) by analyzing their family of fibres. Kwieci\'nski and Tworzewski \cite{KT} introduce a topological invariant $\phi=\phi(f)$  of a mapping $f:X \to Y$, defined as the supremum of numbers of arbitrarily chosen points on any given fibre of $f$ that can be simultaneously approximated by points on general fibres (see Definition~\ref{def:phi} below). They show that, if $Y$ is locally irreducible of dimension $n$, then $\phi(f)$ can only take on values $0,1,\dots,n-1$ and $\infty$, with $\phi(f)=\infty$ being equivalent to openness of $f$. 
We will provide a family of examples (Example~\ref{ex:ell}) which shows that, for every target dimension $n$, all the values of $\phi=0,\dots,n-1$ can be attained by non-open mappings.

If $Y$ is a singular space, the formula given in \cite{KT}  (see \eqref{eq:FFS} below) computes only an \emph{upper} bound for $\phi(f)$. In this paper, we present a similarly effective formula which gives a \emph{lower} bound for $\phi$:

\begin{theorem}\label{th:FFL}
Let $Y$ be a locally irreducible space contained in a space $\Upsilon$ of pure dimension $N$. Let $\Omega$ be a space of pure dimension $k$, and suppose $X \subseteq Y \times \Omega$ is a space which can be defined in $\Upsilon \times\Omega$ locally by at most $r$ holomorphic functions (i.e., every stalk of the coherent ideal of $\mathcal{O}_{\Upsilon \times \Omega}$ defining $X$ admits $r$ generators). Let $f:X \rightarrow Y$ be the projection mapping and suppose that $\phi(f)>0$. Then
$$ \phi(f) \geq \min_{\lambda \neq j \in I} \left[\frac{N-\dim f(X_j)-1}{j-(k-r)}\right],$$
where $I$ is the set of all fibre dimensions of $f$, $\lambda$ is the minimum of $I$, and for every $j \in I$, $X_j$ is the locus of all $\xi \in X$ with $\dim_\xi f^{-1}(f(\xi))=j$. (Square brackets take the integer part of the fraction.)
\end{theorem}

We will comment on the assumptions of Theorem~\ref{th:FFL} in Section~\ref{sec:sing}.

In the special case of mappings with smooth targets, as shown in \cite{KT}, the value of $\phi$ is precisely equal to the upper bound given by formula~\eqref{eq:FFS}. (We have denoted this upper bound by $\FFS$.) In the next two results, which are proved in Section~\ref{sec:special}, we provide a larger class of mappings for which $\phi=\FFS$ holds also.

\begin{corollary}\label{cor:FFL}
Let $Y$ be a locally irreducible, pure-dimensional space, $\Omega$ a space of pure dimension $k$, and
$X$ a pure-dimensional space defined in
$Y \times \Omega$ locally
by $r$ holomorphic functions, where $r=\dim (Y\times\Omega) - \dim X$. (In other words, we may say that $X$ is locally a set-theoretical complete intersection in $Y \times \Omega$).
Then, for the projection $f:X \rightarrow Y$, we have $\phi(f)=\FFS(f)$.
\end{corollary}

\begin{proposition}\label{prop:des}
Consider a mapping  $f:X \rightarrow Y$, 
with $Y$ locally irreducible and of pure dimension.
If $Y$ admits a desingularization which is a finite mapping (in particular,
if $\dim Y=1$),
then $\phi(f)=\FFS(f)$.
\end{proposition}

As an immediate consequence, the important result \cite[Theorem~1.1]{KT} (and, similarly, \cite[Theorem~2.2]{KT}) can be stated for a wider class of mappings than the ones with smooth targets:

\begin{corollary}
Let $f:X \to Y$ be a mapping, with $X$ and $Y$ of pure dimension $d$ and $Y$ locally irreducible, such that either $Y$ admits a finite desingularization or $f$ can be set up as in Corollary~\ref{cor:FFL}. Suppose that exactly for one $\eta_0 \in Y$ the fibre $f^{-1}(\eta_0)$ is of positive dimension. Then, for every $x$ in an open subset of $X$, we have
$$\#f^{-1}(f(x)) \geq \left[\frac{d-1}{\dim f^{-1}(\eta_0)}\right],$$
where $\#$ denotes the number of points.
\end{corollary}

Lastly, we will prove that, given a mapping with a smooth target, the local values of $\phi$ (see Definition~\ref{def:phi_loc}) vary semicontinuously with respect to the (analytic) Zariski topology:

\begin{proposition}\label{prop:semi} 
Let $f: X \to Y$ be a mapping with $Y$ smooth (or, more generally, with $f$ as in Corollary~\ref{cor:FFL} or Proposition~\ref{prop:des}). For every $a \in \N$, the set $\{x \in X \st \phi_x(f) \leq a\}$ is a (closed) analytic subset of $X$. 
\end{proposition}

For the basics about complex analytic spaces and mappings, we suggest \cite{F} and \cite{L}.

\subsubsection*{Some conventions.} We always consider the strong topology of complex analytic spaces (namely, the one induced by Euclidean topology of the local models), unless by mentioning \emph{Zariski topology} we explicitly indicate that the analytic Zariski topology is being considered (namely, the topology defined by taking (Euclidean closed) analytic subsets as closed sets). Since our study concerns only the topological structure, we may always assume that our spaces are reduced. So, in particular, by an irreducible component of a space $X$, we mean an isolated one (i.e., one which is not included in another irreducible component of $X$). Since the nature of our study, and of the invariant $\phi$ in particular, is indeed local, we might implicitly take for granted that, for instance, the number of irreducible components of a space is finite. We clarify that $\N=\{0,1,2,\dots\}$.  

\section{The invariant $\phi$}

We recall a precise definition of $\phi$ (from \cite{KT}) as follows:

\begin{definition}\label{def:phi}
For a mapping $f:X \to Y$, $\phi(f)$ is the supremum of all integers $i \geq 1$ with the following property: for every $\xi \in X$, every $x_1,\dots,x_i$ taken on the fibre $f^{-1}(f(\xi))$, and every subset $B \subseteq Y$ with empty interior, there exist a sequence $\{y_j\}_j$ in $Y \backslash B$, with $y_j \to f(\xi)$, and sequences $\{x_{k,j}\}_j$, $k=1,\dots, i$, such that $x_{k,j} \in f^{-1}(y_j)$ and $x_{k,j} \to x_k$. If no such $i$ exists, then we set $\phi(f)=0$.
\end{definition}

Equivalently (\cite[Proposition~6.2]{KT}), $\phi$ can be defined also as the following:

\begin{definition} \label{def2:phi}
For a mapping $f:X \to Y$, $\phi(f)$ is the largest integer $i \geq 1$ such that the $i$-fold fibred power of $f$, denoted by $f^{\{i\}} :X^{\{i\}} \to Y$, contains no vertical components, and as $\phi(f)=0$ if no such $i$ exists. (We recall that an irreducible component of the source space of a mapping is called \emph{vertical} if its image has empty interior in the target.)
\end{definition}

The latter definition is mainly useful for explicit calculations, particularly considering that vertical components in some important cases (such as the case of algebraic mappings) correspond to torsion elements. (See \cite{ABM}, \cite{A}, and \cite{AS}, besides \cite{KT}, for effective methods of testing for openness in terms of vertical components or torsion elements.)

The formula given in \cite{KT} for estimation of $\phi$ is as follows: Let $f:X \to Y$ be a mapping with $Y$ locally irreducible and of pure dimension $n$. Consider an \emph{equidimensional partition} $\{X_p\}_p$ of $X$ with respect to $f$, which is a locally finite partition of $X$ such that each $X_p$ is a non-empty irreducible (locally closed) subspace of $X$, and such that the restriction $f|_{X_p}$ is an equidimensional mapping (i.e., a mapping whose non-empty fibres are all of pure and the same dimension). Then, $\phi(f)$ is bounded above by 
\begin{equation}\label{eq:FFS}
 \FFS(f) \coloneqq \min_p\{\left[\frac{n-\dim f(
X_{p})-1}{\fbd f|_{X_{p}}-(h_p-n)}\right]
\ \st \ \fbd f|_{X_{p}}> h_p-n \},
\end{equation}
where $\fbd$ stands for fibre dimension, and $h_p=\min\{\dim_\xi X \st \xi \in X_p
\}$. It is proved (\cite[Theorem~3.5]{KT}) that if $Y$ is smooth, then $\phi(f)=\FFS(f)$.

Formula~\eqref{eq:FFS} becomes somewhat simpler if we know that $X$ is of
pure dimension $m$. In such a case, it suffices to first obtain the \emph{partition of $X$ defined by fibre dimensions of $f$}; namely $\{X_j\}_j$, where $X_j$, for every fibre dimension $j$ of $f$, is the locus of all points $\xi \in X$ such that the fibre dimension $\fbd_\xi f$ of $f$ at $\xi$ is equal to $\dim_\xi f^{-1}(f(\xi))=j$. In this case,
\begin{equation}\label{eq:FFS_p}
\FFS(f) =\min_j\{\left[\frac{n-\dim f(
X_{j})-1}{j-(m-n)}\right]
\ \st \ j> m-n \}.
\end{equation}

The following example shows that the classification of non-open mappings by their Kwieci\'nski-Tworzewski invariant $\phi$ is not void; that is, for every target dimension $n$ and every integer $0 \leq i \leq n-1$, there exists a mapping $f:X \to Y$ with $\dim Y=n$ such that $\phi(f)=i$.

\begin{example}\label{ex:ell}
Choose $n,\ell \in \N$, with $1 \leq \ell \leq n$. 
Let $X$ be the analytic subset of $\C^{2n+1}$ with
coordinates $(y_1,\dots,y_n,x_1,\dots,x_{n+1})$, defined by
$$y_1x_1+\cdots+y_\ell x_\ell+x_{n+1}^2=0 \quad \text{and} \quad
y_2x_1+\cdots+y_\ell x_{\ell-1}+y_1 x_\ell=0.$$
(If $\ell=1$, consider the second equation as $y_1x_1=0$.)

Set $Y=\C^n$ with coordinates
$(y_1,\dots,y_n)$. Define $f:X \rightarrow Y$ as
the projection. We claim that $\phi(f)=\ell-1$.

First, we justify that $X$ is of pure dimension $2n-1$, so that we may use \eqref{eq:FFS_p}. If $\ell=1$, then
$X$ is just a union of two ($2n-1$)-dimensional spaces (one of which is clearly vertical over $Y$). So, suppose
that $\ell>1$.\\
For $y\in Y=\C^n$, define $D_y=\begin{bmatrix}
y_1 & \cdots & y_{\ell-1} & y_\ell\\
y_2 & \cdots & y_\ell & y_1
\end{bmatrix}$, and set 
$A=\big\{(y,x) \in X \st \rank D_y < 2 \big\}$. 
We have \label{sym:st}
\begin{align*}
A & =  \big\{ (y,x) \in X \st  \exists\, c \in \C\backslash\{0\} 
\text{\small \ s.t.\ } 
(y_1,\dots,y_\ell)=c(y_2,\dots,y_\ell,y_1)\big\} \\
& \quad \cup \big\{(y,x) \in X \st y_1=\cdots=y_\ell=0\big\} \\
& =  \big\{(y,x) \in X \st 
\exists\, c \in \C\backslash\{0\} \text{\small \ s.t.\ } y_1=c^\ell y_1,
y_2=c^{\ell-1}y_1,\dots,y_\ell = cy_1\big\}\\
& \quad \cup \big\{(y,x) \in X \st y_1=\cdots=y_\ell=0\big\}\\
& = \bigcup_{\substack{c \in \C , c^\ell=1}}
\big\{(y,x) \in X \st y_2=c^{\ell-1}y_1,\dots,y_\ell = cy_1
\big\}.
\end{align*}
By considering the defining equations of $X$, we get
$$
A = \bigcup_{c\in\C,\, c^\ell=1}
\big\{(y,x) \in \C^{2n+1} \st \parbox[t]{7cm}{
$y_2=c^{\ell-1}y_1,\dots,y_\ell = cy_1,$\\
$y_1x_1+c^{\ell-1}y_1x_2+\cdots+cy_1x_\ell+x_{n+1}^2=0,$\\
$c^{\ell-1}y_1x_1+\cdots+cy_1x_{\ell-1}+y_1x_{\ell}=0\big\}.$}
$$
Multiplying the third equation by $c$ gives
$$
A = \bigcup_{c\in\C,\, c^\ell=1}
\big\{(y,x) \in \C^{2n+1} \st \parbox[t]{7cm}{
$y_2=c^{\ell-1}y_1,\dots,y_\ell = cy_1,$\\
$ y_1x_1+c^{\ell-1}y_1x_2+\cdots+cy_1x_\ell+x_{n+1}^2=0,$\\
$ y_1x_1+c^{\ell-1}y_1x_2+\cdots+cy_1x_{\ell}=0\big\},$}
$$
which simplifies to
\begin{equation}\label{eq:A}
A = \bigcup_{c\in\C,\, c^\ell=1}
\big\{(y,x) \in \C^{2n+1} \st \parbox[t]{6cm}{
$y_2=c^{\ell-1}y_1,\dots,y_\ell = cy_1,$\\
$x_{n+1}=0,$\\
$y_1(x_1+c^{\ell-1}x_2+\cdots+cx_{\ell})=0\big\}.$}
\end{equation}
It is now easily seen that $\dim A = 2n-\ell$. Write $X=A\cup
(X \backslash A)$. Since $X$ is defined by two equations, we have
$\dim_\xi X \geq 2n+1-2=2n-1$ for every $\xi \in X$. On the other
hand, $\dim A = 2n-\ell < 2n-1$, and hence $\dim A < \dim_\xi X$
for every $\xi \in X$. 
We conclude that $A$ is a nowhere-dense subset of
$X$.

Take a point $(\eta,\xi) \in X \backslash A$, where $\eta\in\C^n_y$ and $\xi\in
\C^{n+1}_x$. Since $\rank D_\eta=2$,
there is a nonsingular submatrix $\begin{bmatrix}
\eta_i & \eta_j\\
\eta_{\sigma(i)} & \eta_{\sigma(j)} \end{bmatrix}$ for some $i,j$,
$1 \leq i<j \leq \ell$, where $\sigma$ is the permutation
$\begin{pmatrix} 1 & 2 & \cdots & \ell \end{pmatrix}$. It follows
that we can solve the defining equations of $X$ in a neighbourhood
of $(\eta,\xi)$ in $X \backslash A$ for $x_i$ and $x_j$. Hence,
$\dim_{(\eta,\xi)} (X \backslash A) = 2n+1-2=2n-1$, and in particular,
$X \backslash A$ is of pure dimension $2n-1$. It follows that $X$ is of pure dimension $2n-1$.

Next, we need to find the partition of $X$ defined by fibre dimensions of $f$. Take a point $\eta \in Y=\C^n$. If $\rank D_\eta=2$,
then from
defining equations of $X$, we get  
$\dim f^{-1}(\eta)= n-1$.
If $\rank D_\eta < 2$ and $\eta_1 \neq 0$,
then by \eqref{eq:A}, we get $f^{-1}(f(\eta))=\{
(\eta,x) \in \C^{2n+1} \st x_{n+1}=0,
x_1+c^{\ell-1}x_2+\cdots+cx_{\ell}=0\}$. So again, 
$\dim f^{-1}(\eta) = n-1$.
Finally assume $\rank D_\eta < 2$ and
$\eta_1=0$, or equivalently by \eqref{eq:A}, 
$\eta_1=\cdots=\eta_\ell=0$. In this case,
$f^{-1}(f(\eta))=\{(\eta,x) \in \C^{2n+1} \st x_{n+1}=0\}
$, and hence $\dim f^{-1}(\eta)=n$.
Thus, there is only one
non-generic fibre locus to consider: $X_n$, with 
$f(X_n)
=\{\eta_1=\cdots=\eta_\ell=0\}$. One then calculates by \eqref{eq:FFS_p} that
$$\phi(f)=\FFS(f)=\left[\frac{n-(n-\ell)-1}{n-(2n-1-n)}\right]=\ell-1.$$
\tend{example}

\section{Mappings with singular targets}\label{sec:sing}

Theorem~\ref{th:FFL} is the subject of this section. We remark, first, that the assumption of $\phi(f)>0$ is not an important restriction in Theorem~\ref{th:FFL}. Indeed, if $\phi(f)=0$, then the mapping $f$ has a vertical component, the upper bound $\FFS(f)$ readily gives the value of zero, and there will be no need for a lower bound. Second, the special setting of Theorem~\ref{th:FFL} can be established (locally) for any mapping $f:X \to Y$, for example, by embedding $X$ (after shrinking if needed) into some $\Omega=\C^k$, replacing $X$ by the graph $\Gamma_f$ of $f$ in $Y \times X$, and replacing $f$ by the projection $\Gamma_f \to Y$. An (open) question immediately follows:

\begin{question}
Is there a systematic method of setting up (according to assumptions of Theorem~\ref{th:FFL}) a given mapping $f$  such that the best lower bound (and possibly the exact value) for $\phi(f)$ is obtained by Theorem~\ref{th:FFL}?
\tend{question} 

\begin{proof}[Proof of Theorem~\ref{th:FFL}]
For every $\zeta\in X^{\{i\}}$, $i \geq 1$, the germ $X^{\{i\}}_{\zeta}$ will be defined in $(\Upsilon \times \Omega^i)_{\zeta}$ by at most $ir$ equations, according to the underlying set of fibred product. Hence, by estimating the dimension of intersection, for every $\zeta \in X^{\{i\}}$ we can write 
\begin{equation}\label{eq:mi.d}
\dim X^{\{i\}}_{\zeta} \geq \dim(\Upsilon \times \Omega^i)_\zeta-ir = N+i(k-r).
\end{equation}

By considering again the set-theoretical structure of fibred product, write 
$$
X^{\{i\}}=\bigcup_{(j_1,\dots,j_i)}
X_{j_1} \times_Y \cdots \times_Y X_{j_i},
$$ 
where the union is over all $(j_1,\dots,j_i) \in \N^i$ such that $j_1,\dots,j_i$ are fibre dimensions of $f$. Denote by $j_0$ the maximum of $j_1,\dots,j_i$, and consider the projection
$\pi:X_{j_1} \times_Y \cdots \times_Y X_{j_i} \rightarrow X_{j_0}$. Since $\fbd_\zeta (f\circ\pi) \leq ij_0$ for every $\zeta \in X_{j_1} \times_Y 
\cdots \times_Y X_{j_i}$, by the Dimension Formula (see, e.g., \cite[section~V.3]{L}) we get
\begin{equation}\label{eq:ma.d}
\dim(X_{j_1} \times_Y \cdots \times_Y X_{j_i}) \leq ij_0 +\dim f(X_{j_0}).
\end{equation}

Suppose $i \geq 1$ is such that for every fibre dimension $j > \lambda$ (where $\lambda$ is the minimal fibre dimension of $f$) we have 
\begin{equation}\label{eq:ij<}
ij+\dim f(X_j) < N+i(k-r).
\end{equation}
Then, \eqref{eq:mi.d}, \eqref{eq:ma.d}, and \eqref{eq:ij<} imply that $\dim(X_{j_1} \times_Y \cdots \times_Y X_{j_i}) < \dim X^{\{i\}}_{\zeta}$ for every $(j_1,\dots,j_i) \neq (\lambda,\dots,\lambda)$ and  every $\zeta \in X^{\{i\}}$. This implies that $(X_{j_1} \times_Y \cdots \times_Y X_{j_i})_{\zeta}$ is a nowhere-dense subgerm of $X^{\{i\}}_{\zeta}$ for every $(j_1,\dots,j_i) \neq 
(\lambda,\dots,\lambda)$ and  every $\zeta\in X^{\{i\}}$. Now, take a point $\zeta\in X^{\{i\}}$, and write 
$$
X^{\{i\}}_{\zeta}=\Big((X_{\lambda})^{\{i\}}\Big)_{\zeta}
\cup \Big(\bigcup_{(j_1,\dots,j_i) \neq (\lambda,\dots,\lambda)} X_{j_1} \times_Y \cdots \times_Y X_{j_i}\Big)_{\zeta}.
$$
As the second summand is the germ of a locally finite union of nowhere-dense subsets, it follows that $\Big((X_{\lambda})^{\{i\}}\Big)_{\zeta}$ is a dense subgerm of $X^{\{i\}}_{\zeta}$. On the other hand, $f|_{X_{\lambda}}$ is an open mapping (by Remmert's Rank Theorem (see, e.g., \cite[section~V.6]{L}), the fact that $f$ has no vertical components, and local irreducibility of $Y$). Then, since openness is preserved under pulling back, it follows by induction that $(f|_{X_{\lambda}})^{\{i\}} : (X_{\lambda})^{\{i\}} \to Y$ is an open mapping. Thus, $X^{\{i\}}_{\zeta}$ cannot have any vertical components over $Y$. We thus showed that for any $i \geq 1$ such that \eqref{eq:ij<} holds for every $j > \lambda$, $X^{\{i\}}_{\zeta}$
has no vertical components for every $\zeta \in X^{\{i\}}$, and hence $X^{\{i\}}$ has no vertical components. But \eqref{eq:ij<} is equivalent to
$$
i \leq \frac{N-\dim f(X_j)-1}{j-(k-r)}
$$
for every $j>\lambda$. (Note that if $j>\lambda$, then we have $j>k-r$, since by estimation of codimension applied to a fibre, $\lambda \geq \dim\Omega-r=k-r$). The result follows now by Definition~\ref{def2:phi}.
\end{proof}

In the following example, we apply Theorem~\ref{th:FFL} to calculate the exact value of $\phi(f)$.

\begin{example} \label{ex:phi=FFS}
Set
\begin{eqnarray*}
Y
& = & \{y\in\C^4 \st y_1y_4-y_2y_3=0\},\\
X
& = & \{(y,x)\in Y \times \C \st y_1x^2+y_4x+y_2-y_3=0\},
\end{eqnarray*}
and define $f:X \rightarrow Y$ as the projection.

Take a point $\eta \in Y$. If $\eta_1 \neq 0$ or $\eta_4 \neq 0$, then
$f^{-1}(\eta)$ is either a singleton or a pair of points. 
If $\eta_1=
\eta_4=0$, then (by the defining equations of $Y$)
$\eta_2\eta_3=0$, and in order
to get a non-empty fibre, we should have (by the defining equations
of $X$) 
$\eta_3=\eta_4=0$, and so $\eta=0$ and the fibre will be $\C_x$.
Thus, $\{X_0,X_1\}$ is the partition of $X$ defined by fibre dimensions of $f$. We have $\dim f(X_1)=0$.

The space $X$ has pure dimension $3$. By \eqref{eq:FFS_p}, we get 
$$\FFS(f)=\left[\frac{3-0-1}{1-(3-3)}\right]=2.$$
Thus, $\phi(f) \leq 2$.

Now, apply Theorem~\ref{th:FFL} by setting $\Upsilon=Y$, $N=3$, $\Omega=\C$, $k=1$, and $r=1$. We get 
$$\phi(f) \geq \left[\frac{3-0-1}{1-(1-1)}\right]=2.$$

Thus, $\phi(f)=2$.
\tend{example}

\section{Special classes of mappings} \label{sec:special}
In this section, we give the proofs of those results that identify classes of mappings for which we have $\phi=\FFS$.

\begin{proof}[Proof of Corollary~\ref{cor:FFL}]
We can assume that $f$ has no vertical components, as otherwise we have $\phi(f)=\FFS(f)=0$. In the context of Theorem~\ref{th:FFL}, set $\Upsilon=Y$, $n \coloneqq N =\dim Y$, and $m \coloneqq \dim X$. Then, by the Dimension Formula, we have $\lambda=m -n$, where $\lambda$ is the minimum fibre dimension of $f$. On the other hand, by assumption, we have $k-r=m-n$. Now, compare the formula of our lower bound in Theorem~\ref{th:FFL} and formula~\eqref{eq:FFS_p}, to conclude that the former is equal to the latter, and $\phi(f)$ being in between, it follows that $\phi(f)=\FFS(f)$.
\end{proof}

Next, we go over the proof of Proposition~\ref{prop:des}. We need a lemma first.

\begin{lemma}\label{p_sur}
Consider mappings $ f: X \to Y$ and $\sigma:Z \to Y$, and let $ f' : Z \times_Y X \to Z$ be the pullback of $ f$ by $\sigma$. If $\sigma$ is surjective and has no vertical components, then $\phi( f') \leq \phi( f)$.
\end{lemma}

\begin{proof}
Set $i \coloneqq \phi(f)+1$, so that the $i$-fold fibred power $ f^{\{i\}} : X^{\{i\}} \to Y$ has a vertical component, say $\Sigma$. Let $\sigma'$ be the pullback of $\sigma$ by $ f^{\{i\}}$, and let $( f^{\{i\}})'$ be the pullback of $ f^{\{i\}}$ by $\sigma$. We get the following Cartesian square:
$$\begin{CD}
Z \times_Y X^{\{i\}} @>{\sigma'}>> X^{\{i\}}\\
@VV{( f^{\{i\}})'}V @VV{ f^{\{i\}}}V\\
Z @>{\sigma}>> Y
\end{CD}$$
Suppose $\sigma$ is surjective. Then the pullback $\sigma'$ is surjective, and so $(\sigma')^{-1}(\reg(\Sigma))$ is a non-empty open subset of $Z \times_Y X^{\{i\}}$, where $\reg(\Sigma)$ denotes the regular locus of $\Sigma$. Let $\Sigma'$ be an irreducible component of $Z \times_Y X^{\{i\}}$ with a non-empty intersection with $(\sigma')^{-1}(\reg(\Sigma))$. Then $\reg(\Sigma') \cap (\sigma')^{-1}(\reg(\Sigma))$ is a non-empty open subset of $\Sigma'$, which is mapped into the set $\sigma^{-1}( f^{\{i\}}(\Sigma))$ by $( f^{\{i\}})'$. Suppose now $\sigma$ has no vertical components, which implies that the inverse image of a set with empty interior by $\sigma$ has empty interior. Therefore, $\sigma^{-1}( f^{\{i\}}(\Sigma))$ has empty interior in $Z$, as by verticality of $\Sigma$, $ f^{\{i\}}(\Sigma)$ has empty interior in $Y$. Now that an open subset of $\Sigma'$ has an image with empty interior in $Z$, we conclude (by irreducibility of $X$ and thus pure dimensionality of its image) that the whole $\Sigma'$ should have such image. Note that $( f^{\{i\}})'$, by the universal properties of fibred product, is equivalent to $( f')^{\{i\}}$. Thus, $( f^{\{i\}})'$ has a vertical component. By Definition~\ref{def2:phi}, we get $\phi( f') <i$, and thus $\phi( f') \leq \phi( f)$.
\end{proof}

\begin{proof}[Proof of Proposition~\ref{prop:des}]
Let $\sigma: Z \rightarrow Y$
be a finite desingularization of $Y$. Here, by local irreducibility of $Y$, one easily verifies that
$\sigma$ is an open mapping.
Let $ f':Z \times_Y X \rightarrow Z$ be the pullback of
 $ f$
by $\sigma$, and $\sigma':Z \times_Y X \rightarrow X$ the pullback of
$\sigma$ by $ f$. 
Since $\sigma$ is surjective and has no vertical
components, by Lemma~\ref{p_sur}
we get $\phi( f') \leq \phi( f)$.
We know that $\phi( f) \leq \FFS( f)$, and
$Z$ being smooth, we have $\FFS( f')=\phi( f')$. Altogether,
$$\FFS( f')=\phi( f')\leq\phi( f)\leq
\FFS( f).$$
To conclude the Proposition, it suffices to show that
$\FFS( f') \geq \FFS( f)$. 

Let $\{X_j\}_j$ be the partition of $X$ defined by fibre dimensions $j$ of $f$. Note that $\sigma'$ is surjective, for it is a pullback of the surjective mapping $\sigma$. Then observe that, for every $\xi \in X$, the fibre of $f$ through
$\xi$ is isomorphic to the fibre of $f'$ through every $\xi'\in
(\sigma')^{-1}(\xi)$. Therefore, $\{(\sigma')^{-1} (X_j )\}_j$ will be the partition of $Z \times_Y X$
defined by fibre dimensions $j$ of $f'$. By decomposing $X_j$ into irreducible components $\bigcup_i X_{j,i}$, a rank partition $\{X_{j,i} \backslash \bigcup_{k=1}^{i-1} X_{j,k}\}_{j,i}$ of $X$ with respect to $f$ is constructed; denote it $\{X_p\}_p$. In a similar way, construct a rank partition out of $\{(\sigma')^{-1} (X_j )\}_j$; denote it $\{X'_q\}_q$. By \eqref{eq:FFS}, we have
\begin{equation}\label{eq:4}
\FFS(f')=\left[ \frac{\dim Z - \dim f'(X'_q)-1}{j -(h_q-\dim Z)} \right]
\end{equation}
for some $X'_q$, with $\fbd f'|_{X'_q}=j$. Note that, by construction, $X'_q$ is mapped by $\sigma'$ into some $X_p$, with 
\begin{equation}\label{eq:1}
\fbd f|_{X_p}=j. 
\end{equation}
Then $\sigma(f'(X'_q)) \subseteq f(X_p)$, and, as a result of finiteness of $\sigma$, we get
\begin{equation}\label{eq:2}
\dim f'(X'_q) \leq \dim f(X_p).
\end{equation}
Note also that $\sigma'$ is an open, finite mapping (for $\sigma$ is so), and thus $\dim_\xi (Z \times_Y X) = \dim_{\sigma'(\xi)} X$ for every $\xi \in Z \times_Y X$. In particular, 
\begin{equation}\label{eq:3}
h_q=\min\{\dim_\xi  Z \times_Y X \st \xi \in X'_q\} \geq h_p=\min\{\dim_\xi X \st \xi \in X_p\}.
\end{equation}
Now, by \eqref{eq:4}--\eqref{eq:3} and \eqref{eq:FFS}, it follows that $\FFS(f') \geq \FFS(f)$.
\end{proof}

\section{Semicontinuity of $\phi$}

\begin{definition}\label{def:phi_loc}
Let $f:X \to Y$ be a mapping. For an $\xi \in X$, we define \mbox{$\phi_\xi (f) \coloneqq \sup_U \phi(f|_U)$}, where the supremum is over all open subsets $U$ of $X$ containing $\xi$.
\tend{definition}

We finish the paper with the proof of lower-semicontinuity of $\phi$ in Zariski topology.\footnote{The rest of this section is not included in this preprint. For a full version of this article please see \emph{Annales Polonici Mathematici} \textbf{114} (2015), 207--217.}

\subsection*{Acknowledgements} This research was initiated and partially developed in the frame of the doctoral study of the author at Western University. The author sincerely thanks here his PhD supervisor, Professor Janusz Adamus, for his exceptional guidance and kind support during that period, and also, for his valuable comments regarding the organization of this paper.

\end{document}